\documentclass[a4paper, 11pt]{article}

\usepackage[margin=3.2cm]{geometry}
\usepackage[latin1]{inputenc}
\usepackage{amsmath}
\usepackage{amsfonts}
\usepackage{amssymb}
\usepackage{amsthm}
\usepackage{bm}						
\usepackage{mdwlist}					
\usepackage{graphicx}
\usepackage{fancyhdr}
\usepackage{float}
\usepackage{cite}
\usepackage{enumerate}          
\usepackage{graphics}           
\usepackage{fancybox}           

\usepackage{url}
\numberwithin{equation}{section}

\theoremstyle{definition}

\newtheorem{defn}[equation]{Definition}

\newtheorem{rmk}[equation]{Remark}

\theoremstyle{plain}
\newtheorem{thm}[equation]{Theorem}
\newtheorem{prop}[equation]{Proposition}
\newtheorem{lem}[equation]{Lemma}
\newtheorem{cor}[equation]{Corollary}

\renewcommand{\c}{\ensuremath{\mathcal{C}}}

\renewcommand{\d}{\mathrm d}

\newcommand{\E}{\ensuremath{\mathbb{E}}}

\newcommand{\eqdist}{\,\buildrel \rm d \over =\,}

\newcommand{\g}{\ensuremath{\mathfrak{G}}}

\newcommand{\map}{\ensuremath{\longrightarrow}}

\newcommand{\N}{\ensuremath{\mathbb{N}}}

\newcommand{\R}{\ensuremath{\mathbb{R}}}

\renewcommand{\S}{\ensuremath{\mathcal{S}}}

\newcommand{\ul}[1]{\underline{#1}}

\newcommand{\wh}[1]{\widehat{#1}}
\newcommand{\wt}[1]{\widetilde{#1}}

\newcommand{\abs}[1]{\ensuremath{\left|#1\right|}}

\newcommand{\inv}[1]{\ensuremath{\frac{1}{#1}}}

\newcommand{\uind}[2]{\ensuremath{#1^{(#2)}}}

\newcommand{\ts}[1]{\textsc{#1}}

\newcommand{\Om}{\bm\Omega}

\renewcommand{\g}{\ensuremath{\mathcal{L}}}
\newcommand{\regular}{regular }

\title{Product-form invariant measures for Brownian motion with drift satisfying a skew-symmetry type condition}
\author{Neil O'Connell and Janosch Ortmann\\ Warwick Mathematics Institute}

\begin{document}
\maketitle
\fancyhead[R]{\nouppercase{\leftmark}}
\fancyhead[L]{}
\bibliographystyle{acm}
\bibstyle{plain}

\begin{abstract}
	Motivated by recent developments on random polymer models we propose a generalisation of reflected Brownian motion (RBM) in a polyhedral domain. This process is obtained by replacing the singular drift on the boundary by a continuous one which depends, via a potential $U$, on the position of the process relative to the domain. It was shown by Harrison and Williams (1987) that RBM in a polyhedral domain has
an invariant measure in product form if a certain skew-symmetry condition holds.  We show that (modulo technical assumptions)
the generalised RBM has an invariant measure in product form if (and essentially only if) the same skew-symmetry condition holds, independent of the 
choice of potential.  The invariant measure of course does depend on the potential.  Examples include TASEP-like particle systems, generalisations of Brownian motion with rank-dependent drift and diffusions connected to the generalised Pitman transform.
\end{abstract}

\section{Introduction}

There has been much recent development on the study of random polymer models in $1+1$ dimensions
\cite{Sep09, CO'CSZ11, OConnellWarren, OConnellTodaLattice, SepValko10, BorodinCorwin11}. 
These turn out to be closely related to random matrix theory on the one hand and the study of the \emph{KPZ equation} \cite{KPZ}, which was proposed to describe a class of surface growth models, on the other.  An important role is played by an exactly solvable semi-discrete directed polymer model which was introduced in \cite{OConnellYor01} and further studied in \cite{MoriartyOConnell07,SepValko10, OConnellTodaLattice, BorodinCorwin11, Spohn12}. This polymer model can also be viewed as a network of {\em generalised Brownian queues} in tandem (or, equivalently, a TASEP-like interacting particle system).  A important role is played by an analogue of the \emph{output} or \emph{Burke theorem}, which leads to a product-form invariant measure for the series of queues.

Queueing networks \cite{Harrison, Harrison78, Reiman84, Zelicourt81} provide examples of \emph{reflected Brownian motion (RBM)} in a polyhedral domain, introduced and studied by \ts{Harrison} and \ts{Williams} \cite{HW87, Williams87}. The analogue of the output theorem in this more general setting is the existence of an invariant measure in product form. The main result in \cite{Williams87} is that RBM in a polyhedral domain has an invariant measure in product form if and only if a certain \emph{skew-symmetry condition} holds.

Motivated by these facts, we introduce a multidimensional diffusion we call \emph{generalised RBM (GRBM)}. Rather than giving the Brownian motion a singular drift whenever it hits one of the boundaries, we now impose a continuous drift. Its magnitude depends, via a potential $U$, on the position of the process relative to the polyhedral domain. A special case, the \emph{exponentially RBM}, is given by the choice $U(x)=-e^{-x}$, which corresponds to the generalised Brownian queue.
We show that for the GRBM existence of an invariant measure in product form is still equivalent to the skew-symmetry condition of \ts{Harrison--Williams}, irrespective of the function $U$.

By introducing a parameter $\beta$ (which can be viewed as \emph{inverse temperature}), taking $U(x)=-\beta^{-1}e^{-\beta x}$ and letting $\beta\longrightarrow\infty$, we recover the RBM studied by \ts{Harrison--Williams}. In this sense the GRBM process really is a generalisation of reflected Brownian motion. 

An example of GRBM is given by an analogue of the totally asymmetric simple exclusion process (TASEP) with a fixed number of particles.
In the exponential case this corresponds to the above-mentioned generalised Brownian queues in tandem. The existence of an invariant measure in product form allows us to compute the asymptotic speed from equilibrium for this particle system, as well as the speed of the system when started from `$-\infty$' in reverse order (the analogue of the so-called `step initial condition'). 
In the exponential case this particle system can be viewed as a projection of a higher dimensional diffusion on generalised 
Gelfand-Tsetlin patterns which was introduced in \cite{OConnellTodaLattice} and shown to be closely related to the quantum 
Toda lattice.  This higher dimensional diffusion in fact also fits into the framework of GRBM~\cite{OConnell_Whittaker}.

We also consider generalisations of \emph{Brownian motion with rank-dependent drift}, studied by \ts{Pal--Pitman} \cite{PalPitman08} 
and draw connections to the \emph{$\delta$-Bose gas} recently studied and related to reflected Brownian motion in a Weyl chamber 
by \ts{Prolhac--Spohn} \cite{ProlhacSpohn11}.

This paper is organised as follows. In Section~\ref{Sec:PolyDom} we review reflected Brownian motion in a polyhedral domain. Section~\ref{Sec:Analytic} is devoted to an analytic result that corresponds to the basic adjoint relations. The GRBM process is introduced in Section~\ref{Sec:GRBM}, where we also state the main probabilistic results. In Section~\ref{Sec:Examples} we discuss particular cases and examples. Section~\ref{Sec:Proofs} contains proofs.

\section{Reflected Processes in a Polyhedral Domain}
\label{Sec:PolyDom}

We recall here the definition and main properties of the reflected Brownian motion in a polyhedral domain and define its exponential analogue. The former were first introduced and studied by \ts{Harrison} and \ts{Williams} \cite{HW87, Williams87}. A related process was studied even earlier by \ts{Harrison--Reiman}\cite{HarrisonReiman81}, where the domain is always given by an orthant, but the driving Brownian motion is allowed to have a general covariance. Our generalisations cover both settings.

\subsection{RBM in a Polyhedral Domain}

Let us first discuss RBM in a general domain. In the \ts{Harrison--Williams} setting the \emph{polyhedral domain} $G\subseteq \R^d$ in which the process runs is the intersection of $k\geq d$ half-spaces. More precisely let $n_1,\ldots, n_k\in\R^d$ be unit vectors and $b\in\R^k$ then the domain $G$ is defined by
\begin{align*}
	G &= \bigcap_{j=1}^k G_j:= \bigcap_{j=1}^k \left\{x\in\R^d\colon n_j\cdot x_j\geq b_j \right\}.
\intertext{We assume that $G$ is non-empty and that each of the \emph{faces}}
	F_j &= \left\{x\in\overline{G}\colon n_j\cdot x =b_j \right\}
\end{align*}
has dimension $d-1$. In general $G$ may be bounded or unbounded, but we assume that $\{n_1,\ldots,n_k\}$ spans $\R^d$, which means that no line can lie entirely within $\overline{G}$.

The reflections are defined by vectors $q_1,\ldots, q_k\in\R^d$ such that $q_j\cdot n_j=0$ for all $j$. We denote by $N$ and $Q$ the $k\times d$ matrices whose $j^\text{th}$ rows are $n_j$ and $q_j$ respectively. The requirement that $\{n_1,\ldots,n_k\}$ spans all of $\R^d$ is equivalent to existence of an invertible $d\times d$ submatrix $\overline{N}$ of $N$.
 
		\begin{figure}[H]
		\begin{center}
		\includegraphics[scale = .2]{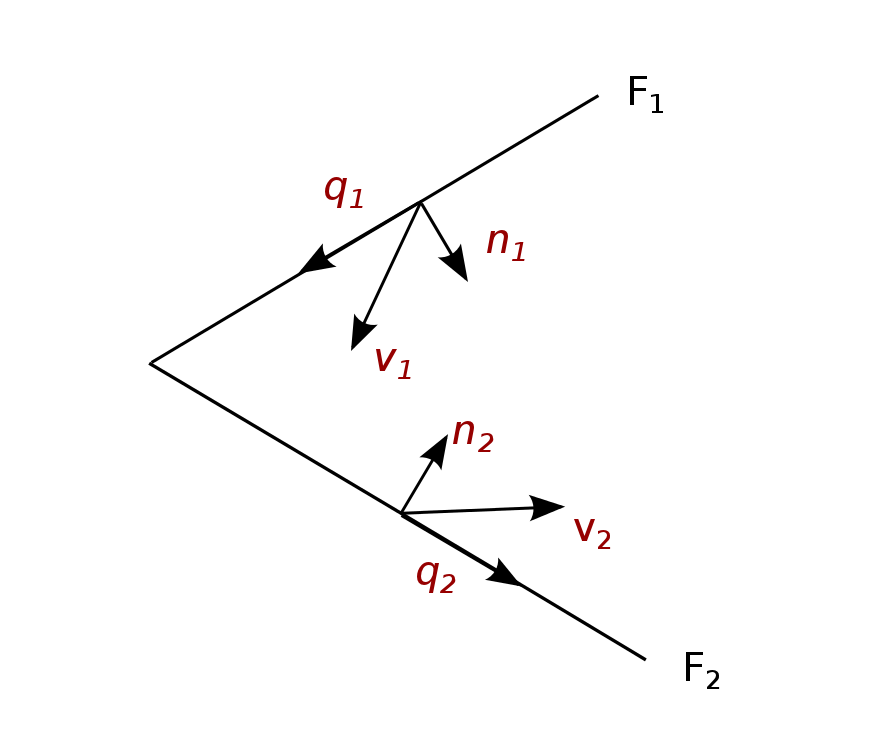}
		\hspace{2cm}
		\caption{Reflection vectors in a wedge}
	\end{center}
		\end{figure}

\par\noindent Informally, reflected Brownian motion $\omega$ in $G$ may be described as follows: inside the domain $G$ the process $\omega$ behaves like a standard Brownian motion with drift $-\mu\in\R^d$, at the boundary it receives an instantaneous singular drift pointing towards the interior -- in direction $v_j:= q_j+n_j$ whenever it hits the face $F_j$ -- and it almost surely never hits any point in the intersection of two or more faces.

In general such a process may not exist. The boundary of the state space is not smooth, and the directions of reflection are discontinuous across non-smooth parts of the boundary, so this does not fit within the Stroock--Varadhan theory \cite{StroockVaradhan71} of multidimensional diffusions. \ts{Williams} \cite{Williams87} showed that if the input data satisfy the \emph{skew-symmetry condition}
\begin{align}
	\label{Eq:HWSkewSymm}
	n_j\cdot q_r + n_r\cdot q_j &=0 \quad\quad \forall j,r\in \ul{k}
\end{align}
then there exists a reflected Brownian motion which can be defined as the unique solution to a submartingale problem. It was also shown in \cite{Williams87} that, under the skew-symmetry condition, reflected Brownian motion has an an invariant measure in product form:



\begin{thm}
	\label{Thm:MainWilliams}
	Suppose that the skew-symmetry condition (\ref{Eq:HWSkewSymm}) holds, then RBM corresponding to $(N,Q,\mu,b)$ has a unique invariant measure whose density with respect to Lebesgue measure is given by
	\begin{align}
		\pi(x) & = \exp\left\{-2\gamma(\mu)\cdot x\right\}
		\intertext{where $\gamma(\mu)$ is defined as follows. By assumption, $N$ has an invertible $d\times d$ submatrix $\overline{N}$. Denote the corresponding submatrix of $Q$ by $\overline{Q}$, then}
		\label{Eq:DefGamma}
		\gamma(\mu) & = \left(I-\overline{N}^{-1} \overline{Q} \right)^{-1}\mu.
	\end{align}
\end{thm}

\begin{rmk}
 	\label{Rmk:Gamma}
 	The existence of an invertible submatrix $\overline{N}$ of $N$ was assumed. By the remarks after equation (1.7) in \cite{HW87} (p. 463) the matrix $(I-\overline{N}^{-1}\overline{Q})$, and hence $\gamma(\mu)$, is independent of the choice of $\overline{N}$, provided the skew-symmetry condition (\ref{Eq:SkewSymm}) holds. Further \cite[(4.7), (7.13)]{HW87} we have $\abs{\gamma(\mu)}^2 = \gamma\cdot \mu$.
 \end{rmk}

\ts{Harrison--Williams} \cite{HW87} also consider reflected Brownian motion in a smooth domain and establish similar results in that setting.

\begin{rmk}
	In two dimensions, the skew-symmetry condition corresponds to a constant angle of reflection across the entire boundary. To be more precise (cf. \cite{HW87}, Example 8.1), fix an orientation on $\partial G$ (which, in two dimensions, is always connected), and let $\tau_j$ be the unit tangent to the face $F_j$ pointing in the positive direction with respect to this orientation. Define the angles $\theta_j\in \left(-\frac{\pi}{2},\frac{\pi}{2}\right)$ by $q_j=\tau_j\,\tan\theta_j$. Then the skew-symmetry condition is equivalent to the requirement that $\theta_j=\theta_r$ for any $r,j$.
	
		\begin{figure}[H]
		\begin{center}
		\includegraphics[scale = .25]{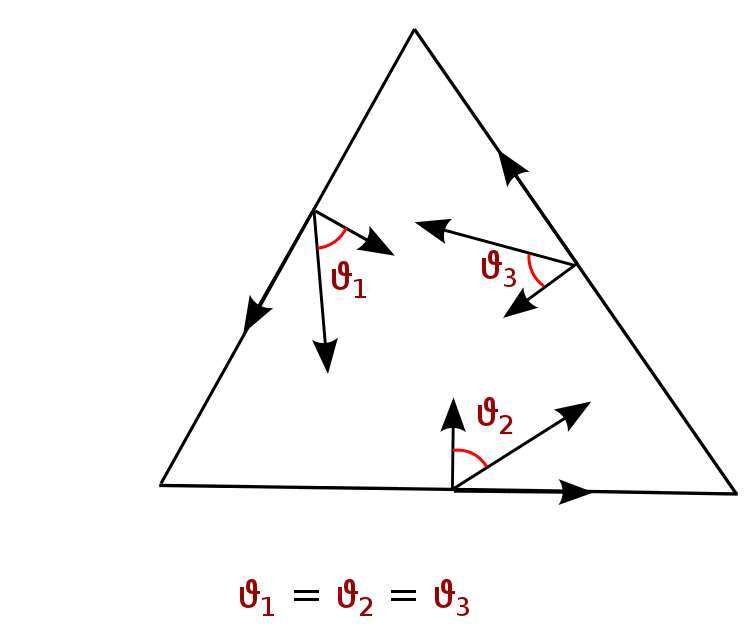}
		\hspace{4cm}
		\caption{The skew-symmetry condition in two dimensions}
	\end{center}
		\end{figure}
\end{rmk}

	
%
%
%

\subsection{RBM in an Orthant}
\label{Sec:RBMOrthant}

We now turn to reflected Brownian motion in the orthant $S=(0,\infty)^d$, driven by a general-covariance Brownian motion. Our definition almost exactly mirrors that of \ts{Harrison--Reiman}  \cite{HarrisonReiman81}. However, we have changed the sign of the reflection matrix $Q$ to make it compatible with the \ts{Harrison--Williams} setup.

Let $d\in\N$ and $B$ be a $d$-dimensional Brownian motion with drift $\mu$ and covariance matrix $A=\sigma\sigma^T$, started inside $S$. That is, there exists a $k$-dimensional standard Brownian motion $\beta$ and a $k\times d$ matrix $\sigma$ with unit rows such that $B(t)=\sigma\beta(t)-\mu t$. In \cite{HarrisonReiman81} additional assumptions on the matrix $Q$ are required, namely that the $d\times d$ matrix $Q$ has non-positive entries, spectral radius less than one and zeroes on the diagonals. \ts{Harrison--Reiman} prove that under these assumptions there exists a unique pair of continuous $\R^d$-valued processes $(Y,Z)$ with
\begin{align*}
	Z(t) & = Y(t) + B(t)(I+Q)
\end{align*}
and such that
\begin{enumerate}[(i)]
	\item $Z(t)\in S$ for all $t\geq 0$
	\item for each $j\in\ul{d}$ the real-valued process $Y_j$ is continuous, non-decreasing and such that $Y_j(0)=0$
	\item $Y_j$ only increases at such times $t$ where $Z_j(t)=0$.
\end{enumerate}
The process $Z$ is called \emph{reflected Brownian motion} in the orthant $S$ with respect to the matrix $Q$, driven by $B$.

If $k=d$ and $\sigma$ is an invertible matrix then it is easy to see that the process $\sigma^{-1}(Z)$ is a reflected Brownian motion in the polyhedral domain $\sigma^{-1}(S)$ in the sense of \ts{Harrison--Williams}.

\section{An Analytic Result}
\label{Sec:Analytic}

Before introducting the GRBM process in the next section we state a purely analytic result. We continue to use the same notation and assumptions for the polyhedral domain $G$ and the normal and reflection matrices $N$ and $Q$. As mentioned in the introduction, the idea is to replace the singular reflection drift by a continuous drift that depends, via a function $U$ on the position of the process relative to the domain $G$. This corresponds to defining a differential operator $\g$ by 
\begin{align}
 	\label{Eq:DefGenerator}
 	\g & = \inv{2}\,\Delta - \mu\cdot\nabla +\sum_{r=1}^k U'\left(n_r\cdot x-b_r\right) v_r\cdot \nabla
\end{align}
and defining the GRBM process as diffusion with generator $\g$. Let now $\nu$ be a measure on $\R^d$ which is absolutely continuous with respect to Lebesgue measure and has density $p$ and such that $\g^* p=0$, where $\g^*$ denotes the formal adjoint of $\g$. That is, $\g^*=\inv{2}\,\Delta-\Om\cdot\nabla - \nabla\cdot\Om$ where
\begin{align*}
		\Om( x) & = \sum_{r=1}^k U'\left(n_r\cdot x- b_r\right)v_r - \mu
\end{align*}
If the operator $\g$ is sufficiently well-behaved (which will depend on the function $U$) this would imply that the measure $\nu$ is invariant with respect to the GRBM process. We begin by the following analytic result, which states that under the skew-symmetry condition (\ref{Eq:HWSkewSymm}) of \ts{Harrison--Williams}, the condition $\g^*p=0$ holds for a particular density function that can be considered the analogue of the function $\pi$ from Section~\ref{Sec:PolyDom}.

We remark that the result does not require $U$ to be anything more than twice continuously differentiable. In the next section we will discuss further conditions on $U$ that ensure a probabilistic interpretation. The partial differential equation $\g^*f=0$ can be viewed as the analogue of the \emph{basic adjoint relations} (BAR) of \cite{HW87, Williams87}.

\begin{thm}
	\label{Thm:Analytic}
	Suppose that the $n_j, q_j$ satisfy the skew-symmetry condition (\ref{Eq:HWSkewSymm}), let $U\colon\R\map\R$ be any twice continuously differentiable function and define the function $p\colon \R^d\map \R$ by
	\begin{align}
		\label{Eq:IMHW}
		p_U(x) & = \exp\left\{2\left(\sum_{r=1}^k U\left(n_r\cdot x - b_r\right)-\gamma(\mu)\cdot x \right) \right\}\, \d x
	\end{align}
	where $\gamma(\mu)$ is as in (\ref{Eq:DefGamma}). Then $\g^*p=0$.
\end{thm}

\par\noindent The proof of this result is deferred to Section~\ref{Sec:Proofs}. A similar result holds for the analogue of RBM in an orthant as defined by \ts{Harrison--Reiman} \cite{HarrisonReiman81} (see Section~\ref{Sec:RBMOrthant}). Here we allow the driving Brownian motion to have a more general covariance matrix, but require that $n_j=e_j$, where $e_1,\ldots,e_d$ denotes the standard orthonormal basis for $\R^d$. Therefore $N=I$, the $d\times d$ identity matrix.

\begin{thm}
	\label{Thm:AnalyticOrthant}
	Let $A=(a_{jk})_{j,k}$ be a symmetric positive definite matrix whose diagonal entries are $a_{jj}=1$ for any $j$. Suppose that $q_1,\ldots,q_d$ satisfy the \emph{modified skew-symmetry condition} 
	\begin{align}
		\label{Eq:SkewSymOrthant}
		 a_{rj}& =\frac{q_{jr} + q_{rj}}{2}\quad\quad\forall\, r\ne j.
		\intertext{Then we have $\g^*p(x)=0$ where $\g^*$ denotes the formal adjoint of the operator}
		\g & = \inv{2}\,\nabla\cdot\left(A\nabla\right) - \mu\cdot\nabla +\sum_{r=1}^k U'\left(x_r\right) \left(q_{r}+e_r \right)\cdot \nabla
		\intertext{and the density function $p$ is defined by}
		\label{Eq:DensityOrthant}
			p(x) &=  \exp\left\{2\left[ \sum_{j=1}^d U\left(x_j\right) - \left(2 A - I-Q\right)^{-1}\mu\cdot x \right]  \right\}.
	\end{align}
\end{thm}

\par\noindent We note that both the original and the modified skew-symmetry condition do not depend on the function $U$.

\section{Generalised Reflected Brownian Motion}
\label{Sec:GRBM}

\subsection{The GRBM process}

We now turn to the probabilistic interpretation of the results from Section~\ref{Sec:Analytic}. Let $B$ be a Brownian motion with drift $-\mu$ and covariance matrix $A=(a_{jk})_{j,k}=\sigma\sigma^T$, whose diagonal entries are $a_{jj}=\alpha>0$ for any $j$. The generator of $B$ is given by $\inv{2}\,\nabla\cdot\left(A\nabla\right)$. We always tacitly assume that either the covariance matrix $A=I$ (the \ts{Harrison--Williams} setting), or $d=k$ and $N=I$ (i.e. the domain is an orthant, generalising the \ts{Harrison--Reiman} RBM). We also assume that the \emph{generalised skew-symmetry condition}
\begin{align}
	\label{Eq:SkewSymm}
	n_j\cdot q_r& + n_r\cdot q_j = \frac{2a_{rj}}{\alpha}\quad\quad \forall\, r\ne j.
\end{align}
holds. Of course equation (\ref{Eq:SkewSymm}) reduces to (\ref{Eq:HWSkewSymm}) and (\ref{Eq:SkewSymOrthant}) in the Harrison--Williams and Harrison--Reiman cases respectively. In order to interpret Theorems~\ref{Thm:Analytic}~and~\ref{Thm:AnalyticOrthant} probabilistically we need the function $U$ to satisfy a certain regularity condition.

\begin{defn}
	A twice continuously differentiable function $U\colon\R\map\R$ is said to be \regular with respect to the input data $(N,Q,\mu,A)$ if
	\begin{enumerate}[(1)]
		\item the second-order differential operator $\g$ defined by
	\begin{align}
	 	\g & = \inv{2}\,\nabla\cdot\left(A\nabla\right) - \mu\cdot\nabla +\sum_{r=1}^k U'\left(n_r\cdot x-b_r\right) v_r\cdot \nabla
	\end{align}
	is the infinitesimal generator of a diffusion process with continuous sample paths in $\R^d$
		\item if $\rho\colon\R^d\map [0,\infty)$ is smooth, integrable and such that $\g^*\rho=0$, then  $\rho(x)\,\d x$ is an invariant measure for this diffusion. Here $\g^*$ denotes the formal adjoint of $\g$.
	\end{enumerate}
\end{defn}

%


\begin{defn}
	\label{Def:RBMGeneral}
	Suppose that the function $U$ is \regular with respect to the input data $(N,Q,\mu,A)$. Then the $\R^d$-valued diffusion with infinitesimal generator $\g$ is called \emph{generalised reflected Brownian motion (GRBM)} corresponding to the potential $U$ and the data $(N,Q,\mu,A)$. 
\end{defn}

We now state a sufficient condition $U$ to be regular with respect to $(N,Q,\mu,A)$ where one of $N$ or $A$ are equal to the identity-covariance case. The proof is deferred to Section~\ref{Sec:Proofs}.

\begin{prop}
	\label{prop:AssumptionU}
	Suppose either $A=I$ or $k=d$ and $N=I$, that the $n_j, q_j$ are as above, in particular satisfying the skew-symmetry condition (\ref{Eq:SkewSymm}) and that $U$ satisfies the following conditions:
	 	\begin{enumerate}[(i)]
	 	 		\item $x-U(x)\longrightarrow \infty$ as $\abs{x}\longrightarrow \infty$
	 	 		\item For each $a\in\R$ there exists $\theta_a>0$ such that $U(x+a)-U(a)\leq \theta_a\, x$ for all $x\in\R$.
	 	 		\item There exist $\kappa,\alpha>0$ such that $\gamma_r U'(x)-\inv{2}\,U''(x)\leq \alpha \left(\theta_{-b_r} x - U(x)\right)+\kappa$ for all $x\in\R$ and $r\in\ul{k}$. Here, $\gamma_r = \sum_{s=1}^k \theta_{-b_s} \left(v_r\cdot n_s\right)+\mu\cdot n_r$.
	 	 	\end{enumerate}
	Then $U$ is regular with respect to $(N,Q,\mu,A)$.
\end{prop}

\par\noindent If $N,Q$ satisfy the skew-symmetry condition then the choice $U(x)=-e^{-x}$ satisfies the hypotheses of Proposition \ref{prop:AssumptionU}, and hence is regular with respect to $(N,Q,\mu,A)$. We refer to the process with this choice of $U$ as the \emph{exponentially reflected Brownian motion (ERBM)} corresponding to $(N,Q,\mu,A)$.

\begin{rmk}
	\label{Rmk:kSmallerd}
	Apart from the requirement that the $n_j$ contains a basis of $\R^d$ one could allow $k<d$, that is fewer half-spaces than the dimension we are in. We can then decouple the `superfluous' dimensions as follows. Denote by $E_1\cong \R^k$ the span of $\{n_1,\ldots,n_k\}$ and by $E_2$ its orthogonal complement in $\R^d$. Let further $P_j$ be orthogonal projection from $\R^d$ onto $E_j$. By orthogonal invariance of Brownian motion it follows that $P_1(X)$ and $P_2(X)$ are independent. Further $P_2(X)$ is just a standard $(d-k)$-dimensional Brownian motion, whereas $P_2(X)$ is generalised RBM in $\R^k$ with the data $(P_1(N), P_1(Q), P_1(\mu),I)$.
	
With this in mind we will assume throughout that $k\geq d$.
\end{rmk}

\subsection{Invariant measures for GRBM in a general domain}

Having defined the GRBM process let us now establish the probabilistic interpretations of Theorems~\ref{Thm:Analytic}~and~\ref{Thm:AnalyticOrthant}. We first consider the identity-covariance case $A=I$. The main result is that, under the skew-symmetry condition, generalised reflected Brownian motion has an invariant measure in a certain product form. It follows directly from Theorems~\ref{Thm:Analytic}

\begin{cor}
	\label{Thm:MainGeneralPotential}
	Suppose that $U$ is regular with respect to $(N,Q,\mu,I)$ and that the $n_j,q_j$ satisfy the skew-symmetry condition (\ref{Eq:SkewSymm}). Then the $\R^d$-valued diffusion with generator $\g$ as defined in (\ref{Eq:DefGenerator}) has as invariant measure of the form $\nu_U(\d x) = p(x)\, \d x$, where $p$ is as in (\ref{Eq:DensityOrthant}).
\end{cor}

 \begin{rmk}
     	\label{Rmk:GradientDiff}
     	Under normal reflection (i.e. $Q=0$) the GRBM process is a \emph{gradient diffusion}, that is its generator is of the form $\inv{2}\,\Delta+\nabla W\cdot\nabla$. Note that $Q=0$ implies $\gamma(\mu)=\mu$. So in this case our result also follows from the well-known fact (see \cite{Hairer_LN} for example) that a gradient diffusion has $e^{2W(x)}\,\d x$ as invariant measure.
     \end{rmk}

\begin{rmk}
	\label{Rmk:Ergodicity}
	If the function $U$ satisfies the condition
		\begin{align*}
			-\inv{2}\,U''(x) + \left(1+\mu\right)U'(x)^2 - \mu\leq -\gamma 
		\end{align*}
		for all $x\in\R^d$ and the corresponding density $p$ is integrable then the stationary distribution obtained by suitably normalising (\ref{Eq:IMHW}) is actually unique. For example if $k=d$ and $U(x)=-e^{-x}$, then there is a unique invariant measure if and only if $\mu\cdot n_j> 0$ for all $j$.
\end{rmk}

\par\noindent Applied to the special case $U(x)=-e^{-x}$, Theorem~\ref{Thm:MainGeneralPotential} gives the invariant measure for exponentially reflected Brownian motion.

\begin{cor}
	\label{Thm:MainGeneralDomain}
	Suppose that the skew-symmetry condition (\ref{Eq:SkewSymm}) holds and the measure
	\begin{align*}
	\nu(\d x) & = \exp\left\{-2\left(\gamma(\mu)\cdot x + \sum_{j=1}^k e^{b_j-n_j\cdot x}\right) \right\} dx
\end{align*}
is finite.
	Then the exponentially reflecting Brownian motion corresponding to $(N,Q,\mu,I)$ has an invariant measure $\nu$.
\end{cor}

\begin{rmk}
When $d=k$ then $\nu$ can be realised as the distribution of a $\R^d$-valued random variable $\chi$ such that 
\begin{align*}\
	\left(n_1\cdot \chi,\ldots,n_d\cdot \chi\right) &\eqdist \left(-\log\left(\frac{\gamma_1}{2} \right),\ldots,-\log\left(\frac{\gamma_d}{2} \right) \right)
\intertext{where the $\gamma_j$ are independent gamma random variables with parameters $\theta_j$ and the vector $\theta$ is given by}
	\theta & =2\left(N^T-Q^T\right)^{-1}\mu.
	\intertext{That is, the law of $n_j\cdot \chi$ is given by $\Lambda_{\theta_j}$, where for $\alpha>0$,}
	\Lambda_\alpha(\d x)&= \inv{Z}\,\exp\left\{-2\left(\alpha x + e^{-x}\right) \right\}\,\d x
\end{align*}
 
\end{rmk}

\begin{rmk}
	\label{Rmk:RBMScaling}
	Let $\beta>0$ and set $U_\beta(x)=-\inv{\beta}\, e^{-\beta x}$ for $\beta>0$. The diffusion with generator $\g$ corresponding to $U=U_\beta$ should converge, as $\beta\to\infty$ to the Harrison--Williams reflected Brownian motion. Moreover (cf. \cite{CO'CSZ11}, section 4.1) the log-gamma random variables converge to the exponential distribution, and we recover the main result of \cite{Williams87}. In this sense our results can be considered as a generalisation of those of \cite{HW87, Williams87}.
\end{rmk}

\subsection{Invariant measure for GRBM in an orthant}

We now turn to exponential Brownian motion in an orthant, driven by a $d$-Brownian motion with drift $\mu$ that is allowed to have a general, possibly singular covariance. It can be realised as $B(t) = \sigma\beta(t)$ for a standard Brownian motion $\beta$, possibly of a different dimension, and a matrix $\sigma$, which is generally rectangular. The covariance matrix $A=(a_{jk})=\sigma \sigma^T$ of $B$ has $1$ on all its diagonal entries (or, equivalently, the rows of the rectangular matrix $\sigma$ have unit length). This requirement can be weakened slightly, see Remark~\ref{Rmk:ScalingA}.

In this case the generator of the generalised RBM is given by
\begin{align*}
	\g &= \inv{2}\,\nabla\cdot\left(A\nabla\right) +\left[\sum_{r=1}^d \, U'\left(x_r\right) \left(q_{r}+e_r \right)- \mu  \right]\cdot \nabla.
\end{align*}
The probabilistic interpretation of Theorem~\ref{Thm:AnalyticOrthant} is the following

\begin{cor}
	\label{Thm:Orthant}
	Suppose that $U\colon\R\map\R$ is regular with respect to $(I,Q,\mu,A)$ and that for all $j\ne r$ the \emph{modified skew-symmetry condition}
	\begin{align}
		q_{jr} + q_{rj} & =2a_{rj}\quad\quad\forall\, r\ne j
	\end{align}
	holds. Then the GRBM in an orthant corresponding to $U$ has an invariant measure whose density with respect to Lebesgue measure is given by
	\begin{align}
	p_U(x) &=  \exp\left\{2\left[ \sum_{j=1}^d U\left(x_j\right) - \left(2 A - I-Q\right)^{-1}\mu\cdot x \right]  \right\}
\end{align}
provided that $\int p_U<\infty.$
\end{cor}

\begin{rmk}
	\label{Rmk:ScalingA}
	If the rows of $A$ are not unit vectors but all have the same Euclidean norm $\sqrt{\alpha}$ then we can scale the whole system so that our reasoning still applies. In this setting we get the stationary density in product form
\begin{align}
	\label{Eq:SkewSymOrthantScaled}
	p(x)&= \exp\left\{ 2\left[ \sum_{j=1}^d U \left(\frac{x_j}{\sqrt{\alpha}}\right) - \left(\frac{2}{\alpha}\,A -(I+Q)\right)^{-1}\mu\cdot x \right]\right\}.
\end{align}
\end{rmk}

\section{Examples}

\label{Sec:Examples}

\subsection{One-dimensional ERBM and Dufresne's identity}

As a warm-up let us consider one-dimensional exponentially RBM. Here we can use a particular realisation of the process and Dufresne's identity. 

In this situation $n=1$ and $q=0$. All conditions, including skew-symmetry (\ref{Eq:SkewSymm}), are satisfied. Let further $\mu>0$. The generator of $X$ in this simple case is given by
\begin{align*}
	\uind{\g}{\mu}_1 & = \inv{2} \frac{\d^2}{\d x^2} + \left(e^{-x}-\mu\right) \frac{\d}{\d x}.
\intertext{By It\^o's formula and stochastic integration by parts \cite{RY, KaratzasShreve} the process $X$ given by}
		X(t) & = \log \int_0^t e^{\uind{B}{\mu}(s)-\uind{B}{\mu}(t)}\, \d s
	\end{align*}	
	is a diffusion with generator $\uind{\g}{\mu}_1$. The invariant measure of $X$ is that of $\eta=\log(\xi)$ where $\xi\eqdist 4\uind{A}{2\mu}_\infty$ and the process $\uind{A}{\mu}$ is defined by \cite{Dufresne01Asian, Dufresne01Geometric}
\begin{align*}
	\uind{A}{\mu}_t & =\int_0^t e^{2(B(s)-\mu s)}\, \d s.
\end{align*}
Recall Dufresne's identity \cite[Corollary 4]{Dufresne01Asian}. 

\begin{prop}
	Let $\mu>0$, then
\begin{align*}
	\left(2\uind{A}{\mu}_\infty\right)^{-1} & \eqdist \gamma_\mu
\end{align*}
where $\gamma_\mu$ has the Gamma distribution with parameter $\mu$.
\end{prop}

\par\noindent Hence $\xi\eqdist \frac{2}{\gamma_{2\mu}}$ and so the invariant distribution of the process $X$ is realised by
\begin{align*}
	\eta & = - \log \xi^{-1} \eqdist -\log \left(\frac{\gamma_{2\mu}}{2}\right).
\end{align*} 
I.e. our results recover Dufresne's identity. Let us also note that, replacing the potential $e^{-x}$ by $\inv{\beta}e^{-\beta x}$ (as in Remark~\ref{Rmk:RBMScaling}) it can be verified directly that both the one-dimensional reflecting Brownian motion and the exponential distribution appear in the scaling limit as $\beta\to\infty$.

\subsection{TASEP-like particle systems}

\label{Sec:TASEPlike}

Consider the following particle system on the line. There are $n$ particles, with positions $X_1(t),\ldots,X_n(t)$ at time $t$, labelled such that $X_1(0)\leq X_2(0)\leq \cdots\leq X_n(0)$. The left-most particle $X_1$ evolves like a Brownian motion with drift $\nu_1$, the second particle $X_2$ like a Brownian motion with drift $\nu_2$ reflected off $X_1$ and so on. This particle system, first introduced by \ts{Glynn--Whitt} \cite{GlynnWhitt91}, can be thought of as a Brownian analogue of the totally asymmetric simple exclusion process (TASEP).

Replacing the `hard' reflection by a `soft' one, determined by the potential $U$ we obtain a diffusion on $\R^n$ which can be defined as the unique strong solution of the system of SDEs
\begin{align}
	\notag
	\d X_1(t) & = \d B_1(t)\\
	\label{Eq:SDETASEP}
	\d X_{k+1}(t) & = \d B_{k+1}(t) - U\left(X_{k+1}(t)-X_{k}(t)\right)\,\d t
	\end{align}
where $B=(B_1,\ldots,B_n)$ is an $n$-dimensional Brownian motion with drift $\nu\in\R^n$. The choice of $U(x)=-e^{-x}$ was first considered in \cite{OConnellYor01}. By viewing the process as an analogue of a queueing system and applying a suitably generalised version of Burke's theorem a stationary distribution for the differences $Y_j=X_{j+1}-X_j$ is obtained. In particular, if the drift of the driving Brownian motion is given by $\nu=(\alpha,0,\ldots,0)$ then the stationary distribution of $Y$ is given by $\bigotimes_{j=1}^{n-1}\Lambda_\alpha$. By the law of large numbers, it follows that the speed of the particle system started in equilibrium is given by
\begin{align}
	\label{Eq:MPDExp}
	\Psi(\alpha) & = \lim_{n\to\infty} \inv{n}\, X_n(n) = \E\zeta_1 = -\, \frac{\Gamma'(\alpha)}{\Gamma(\alpha)}
\end{align}
where $\Gamma$ is the Gamma function and $\Gamma'$ its derivative, the digamma function. This allows one to compute the free energy of a related semi-discrete directed polymer model. Suppose now that the drift of the driving Brownian motion $B$ is zero and define
\begin{align*}
Z_n(t) & = \int_{\Delta_n(t)} \exp\left\{B_1(s_1)+ B_2(s_2)-B_2(s_1) + \ldots + B_n(t) - B_n(s_{n-1}) \right\}\, \d s_1\ldots\d s_{n-1}
\end{align*}
where $\Delta_n(t)=\{(s_1,\ldots,s_{n-1})\in [0,\infty)\colon s_1\leq\ldots\leq s_{n-1}\leq t\}$. It was shown in \cite{OConnellYor01} that,
almost surely,
\begin{align}
	\notag
\gamma(\alpha) : = -\left(-\Psi\right)^*(\alpha) &= \lim_{n\to\infty} \inv{n}\log Z_n(\alpha n)
\intertext{where $\cdot^*$ denotes convex conjugation, i.e. $f^*(\alpha)=\sup\{\alpha x - f(x)\colon x\in\R\}$. Note that $\left(\log Z_k\right)_{k=1}^n$ satisfies (\ref{Eq:SDETASEP}) for the exponential case, with a particular entrance law. Hence, if we start the particle system $X_k$ with this entrance law (which corresponds in an appropriate sense to the `step initial condition' $X_1(0)=0$ and $X_{j+1}(0)-X_j(0)=-\infty$ for each $j$) 
we can compute the speed of the system: almost surely,}
	\label{Eq:SpeedExp}
	 \lim_{n\to\infty} \inv{n}\,X_n(n\alpha) &= \gamma(\alpha) .
\end{align}
For a general potential $U$, the gaps between the particles, defined by $Y_j(t)=Y_j(t)-Y_{j+1}(t)$ evolve as a GRBM in the $n-1$-dimensional orthant; hence, by Corollary~\ref{Thm:Orthant}, $Y$ has an invariant measure given by
\begin{align}
	\notag
	\xi_U(\d x) &= \inv{Z_U}\,\prod_{k=1}^{n-1} \exp\left\{-2\left( U(x_k) + \left(\nu_k-\nu_1 \right) x_k\right) \right\}\, \d x_k
\intertext{where we assume that there exists $Z_U\in(0,\infty)$ such that $\xi_U$ is a probability measure. So then it is still true that $Y$ has a stationary distribution in product form, and under the special drift $\nu=(\alpha,0,\ldots,0)$ the marginals are identically distributed. Hence the mean particle density of the system started in equilibrium is given, almost surely, by}
	\label{Eq:MPDGen}
	\Psi_U(\alpha) & = \lim_{n\to\infty} \inv{n}\, X_n(n) = \inv{Z_U}\int_{\R}\,x\exp\left\{-2 U(x)+2\alpha x \right\}\, \d x.
\end{align}
provided that this is finite. Hence, under reasonable assumptions on $U$, it should be the case that, from an appropriate entrance law corresponding to the `step
initial condition' $X_1(0)=0$ and $X_{k+1}(0)-X_{k}(0)=-\infty$ for each $k$, the speed of the system will be given by $-(-\Psi_U)^*$.

\subsection{Examples related to the Pitman transform}

The special case of $U(x)=-e^{-x}$ in Section~\ref{Sec:TASEPlike} is also related to the exponential analogue of the Pitman transform \cite{BBO_Crystal, BBO_Littelman}. This transform is defined, for $\alpha\in\R^n$ and $\eta\in\c\left((0,\infty); \R^n\right)$ by
\begin{align}
	\label{Eq:DefGPT}
	T_\alpha \eta (t) &= \eta(t)+\left(\log \int_0^t e^{-\alpha^\vee\cdot \eta(s)}\,\d s\right)\alpha,
	\intertext{where $\alpha^\vee=2\alpha/|\alpha|^2$.
Suppose now that $\eta$ is a $n$-dimensional Brownian motion. 
For a sequence of vectors $\gamma_1,\ldots,\gamma_q\in\R^n$ and $j\in\{1,\ldots,q\}$ define}
	\label{Eq:DefGPTH}
	\eta_{j} & =  T_{\gamma_j}\circ\cdots\circ T_{\gamma_1}(\eta) = T_{\gamma_{j-1}}\left(\eta_{j-1}\right)
	\intertext{and then set}
	\label{Eq:DefGPTY}
	y_j & = \log \int_0^t e^{\gamma_{j}^\vee\cdot\left(\eta_{j-1}(t)-\eta_{j-1}(s)\right)}\,\d s.
\end{align}
We will see below (Proposition~\ref{Prop:PitmanQueue}) that the process $(y_1,\ldots,y_q)$ fits in the framework of GRBM and hence has a product form invariant measure.

Of particular interest is the following special choice of $\gamma_1,\ldots,\gamma_q$. 
Let $\alpha_j=e_j-e_{j+1}$, $j=1,\ldots,n-1$,
where $e_1,\ldots,e_n$ denotes the standard orthonormal basis of $\R^n$.
Let $S(n)$ denote the symmetric group on $n$ elements and denote the adjacent transpositions
by $s_j=(j\ \ j+1)\in S(n)$. Under the natural action of $S(n)$ on $\R^n$, the $s_j$ correspond 
to reflections in the orthogonal hyperplane to $\alpha_j$. Let
\begin{align}
	\notag
	\sigma_0 & = \begin{pmatrix}
		1& 2&\ldots& n\\ n & n - 1&\ldots&1
	\end{pmatrix}
	\intertext{be the longest element in $S(n)$, which has the following reduced decomposition of length $q=\frac{n(n-1)}{2}$ }
	\notag
	\sigma_0 & = s_1\ldots s_{n-1}\, s_1\ldots s_{n-2}\ \ldots\ s_1 s_2\,s_1.
	\intertext{Let $\gamma_1,\ldots,\gamma_q\in\R^n$ be defined by}
	\label{Eq:SpecialGamma}
	\gamma_1\ldots\gamma_q & =\alpha_1\ldots\alpha_{n-1}\, \alpha_1\ldots\alpha_{n-2}\ \ldots\ \alpha_1\alpha_2\,\alpha_1
\end{align} 
and define the process $y=(y_1,\ldots,y_q)$ as in (\ref{Eq:DefGPTH}) and (\ref{Eq:DefGPTY}). Then $y$ has an invariant measure with density
\begin{align*}
	&\exp\left\{-\left[2\sum_{j=1}^m e^{-x_j/\sqrt{2}}+ 2 \theta\left(\mu\right)\cdot x\right] \right\}
\end{align*}
provided that $\mu\cdot \alpha_r>0$ for all $r$. Here $\theta(\mu)$ is a parameter which
will be defined in Proposition~\ref{Prop:PitmanQueue}, see (\ref{Eq:DefBeta}).

It follows from the definition of the $y_i$ that the stochastic evolution of the process $\left(y_1,\ldots,y_{n-1}\right)$ is identical in law to the inter particle distances in the
TASEP-like process $X$ from Section~\ref{Sec:TASEPlike} given the appropriate
relation between the drifts $\mu$ and $\nu$. 
On the other hand,
for any $\sigma\in S(n)$ and any reduced decomposition $\sigma = s_{j_1}\ldots s_{j_r}$ the operator $T_\sigma:= T_{\alpha_{j_r}}\circ\ldots\circ T_{\alpha_{j_1}}$ only depends on $\sigma$, not the choice of decomposition \cite{BBO_Littelman}. It was shown in \cite{OConnellTodaLattice} that the process 
$\left(\left(T_{\sigma_0}\eta\right)(t)\right)_{t\geq 0}$ is a diffusion which
is closely related to the quantum Toda lattice, and can be thought of as a positive-temperature analogue of Dyson's Brownian motion. This allows one to compute the Laplace transform of the partition function of the semi-discrete directed polymer mentioned above. 
The process $(y_1,\ldots,y_q)$ can also be interpreted in terms of 
a natural diffusion on the set of totally positive lower triangular matrices~\cite{OConnell_Whittaker}.
The fact that this process has an invariant measure in product form can be viewed as a 
positive temperature analogue of Proposition 5.9 in \cite{BBO_Crystal}.

The appearance of product-form invariant measures in this context 
does not depend on the algebraic structure described above. In particular the above example has a natural generalisation which replaces $S(n)$ with an arbitrary finite reflection group.  
We remark that a study of diffusions analogous to $T_{\sigma_0}(\eta)$ for the generalised 
quantum Toda lattice (corresponding to arbitrary finite reflection groups) can be found in 
\cite{Reda_Thesis}.

The following proposition describes the general setting.

\begin{prop}
		\label{Prop:PitmanQueue}
	Let $\gamma_1,\ldots,\gamma_q$ be any collection of vectors in $\R^m$ such that $\abs{\gamma_j}^2=2$ and assume that $U\colon\R\map\R$ is regular with respect to $(I,\Gamma,\mu,A)$ where $A_{jk}=\gamma_{j}\cdot \gamma_k$ and the matrix $\Gamma$ is defined by 
	\begin{align}
		\notag
		\Gamma_{jk} & =
		\begin{cases}	
			0 & \text{if } j\geq k\\
			\gamma_j\cdot\gamma_k\quad\quad & \text{otherwise.}
		\end{cases}
		\intertext{Suppose further that $y=(y_1,\ldots,y_q)$ satisfies the system of SDEs}
		\label{Eq:ExSDEs}
		\d y_k(t) &= \begin{cases} 
			\d\left(\gamma_1\cdot \eta(t)\right) +U'\left(y_1(t)\right)\, \d t & \text{if } k=1 \\
				    \d\left(\gamma_k\cdot \eta(t)\right) +  \left[ \sum_{j=1}^{k-1} \left(\gamma_k\cdot\gamma_j\right)\, U'\left(y_j(t)\right) + U'\left(y_k(t)\right)\right]\, \d t\quad & \text{if } k>1
			\end{cases}
\end{align}
and that $\theta_r\cdot\mu>0$ for all $r$, where $\theta_r(\mu)=-\beta_r\cdot \mu$, the parameters $\beta_r$ are given by
\begin{align}
	\label{Eq:DefBeta}
	\beta_k & = \begin{cases}
		\gamma_1 & \text{if } k=1\\
		\wt{s}_{\gamma_1}\cdots \wt{s}_{\gamma_{k-1}}(\gamma_k)\quad & \text{if } k>1.
		\end{cases}
	\intertext{and the $\wt{s}_\gamma$ are defined by $\wt{s}_\gamma(x)=x-\frac{2\gamma\cdot x}{\abs{\gamma}^2}\,\gamma$. Then the $q$-dimensional diffusion $y$ has an invariant measure with density}
		&\exp\left\{2\sum_{j=1}^q U\left(\frac{x_j}{\sqrt{2}}\right) - 2 \theta\left(\mu\right)\cdot x \right\}.
	\end{align} 
\end{prop}

\begin{proof}
	The process $y$ is GRBM in a $q$-dimensional orthant, driven by the process $\eta$ given by
	\begin{align*}
		\eta(t) &= \begin{pmatrix}
		\gamma_1^T\\ \vdots\\ \gamma_q^T
	\end{pmatrix}\
	\beta(t)
		\end{align*}
with $\beta$ a standard Brownian motion. The covariance matrix of $\eta$ is therefore given by $A$, while the reflection matrix is $\Gamma$. It is now straightforward to verify that the modified skew-symmetry condition holds. An application of Corollary~\ref{Thm:Orthant} completes the proof.
	\end{proof}

%
%

Even in one dimension, Proposition~\ref{Prop:PitmanQueue} yields non-trivial and interesting examples. For example, take $U(x)=-e^{-x}$ and choose $\gamma_j=(-1)^{j+1}\,\sqrt{2}$. The corresponding system of SDEs for the $y_j$ is given by
\begin{align}
	\label{Eq:PitmanQueueOne}
	\d y_1(t) & = \sqrt{2}\,\d B(t) - e^{-y_1(t)}\,\d t\\
	\d y_k(t) & = (-1)^{k+1}\sqrt{2} \,\d B(t) +\left[2 \sum_{j=1}^{k-1} (-1)^{j+k}e^{-y_j(t)} - e^{-y_1(t)}\right]\,\d t.
\end{align}
A realisation of this sequence of one-dimensional processes is given by the following construction. Define the operator $\S$ on continuous real-valued paths on by
\begin{align*}
	\left(\S\eta\right)(t) & = \eta(t) - 2\log\int_0^t e^{-\eta(s)}\,\d s.
	\intertext{Then we have}
	\left(\S B\right)(t) & = B(t)-2 J_1(t) + 2 J_1(0)
	\intertext{where $J_1$ is the exponential analogue of th future infimum in the Pitman theorem, i.e.}
	J_1(t) & = -\log\int_t^\infty e^{-B(u)}\,\d u	.
	\intertext{If we define inductively sequences $\left(Y_n\colon n\in\N\right)$ and $\left(J_n\colon n\in\N\right)$ of processes taking values in $\R$ by $Y_1(t)=\S B(t) + J_1(t)$ and further}
	J_{n+1}(t) &= -\log \int_t^\infty e^{-\S^n B(u)}\,\d u\\
	Y_{n+1}(t) &= \S^{n+1} B(t) + J_{n+1}(t)
\end{align*}
then $Y_1,Y_2,\ldots$ satisfies the system of SDEs (\ref{Eq:PitmanQueueOne}). Thus, $J_1(0), J_2(0), J_3(0), \ldots $ is an i.i.d sequence of random variables and $\S$ acts on it by left shift. 
In particular, if the Brownian motion path $\left(B(t),\ t\geq 0\right)$ could be recovered almost
surely from the sequence $J_1(0), J_2(0), J_3(0), \ldots $, then the system would be metrically isomorphic to a Bernoulli shift.  We leave this as an open question.  A discrete version of this open question was answered in the affirmative in \cite{KeaneOConnell08}.

\subsection{Generalised Pitman transform for other functions}

The fact that invariant measures in product form appear irrespective of the choice of potential suggests generalising the exponential analogue of the Pitman transform to other functions than the exponential, as follows.

Suppose that $U\colon\R\map \R$ is such that, for smooth $\eta\colon (0,\infty)\map\R$ the differential equation
\begin{align*}
	\d\left(\uind{T}{U}(\eta)\right)(t) & = \d \eta(t) + U'\left(\uind{T}{U}\left(\eta\right)(t)\right)\,\d t\\
	\intertext{has a unique solution with $\uind{T}{U}(0+)=-\infty$. This is true
	for $U(x)=-e^{-x}$, and for this choice of U we have $\uind{T}{U}=T_{\sqrt{2}}$, where $T_\alpha$ is the generalised Pitman transform as defined in (\ref{Eq:DefGPT}). This is extended to higher dimensions by noting that the (exponential) generalised Pitman transfrom $T_\alpha$ only acts on the linear span of $\alpha$ and leaves the orthogonal space $\alpha^\perp$ invariant. So we can define the generalised Pitman transform as follows. Write $\Pi_\alpha$, $\Pi_{\alpha^\perp}$ for the orthogonal projections from $\R^d$ to the span and orthogonal complemement of $\alpha$ respectively and define $T_\alpha^{(U)}$ by}
	T_\alpha^{(U)}\eta(t) & = \Pi_{\alpha^\perp}\eta(t) + T^{(U)}\left(\Pi_{\alpha}\eta(t)\right)\,\alpha.
\end{align*}
It could be interesting to study the algebraic properties of these operators, as in \cite{BBO_Crystal, BBO_Littelman}.

\subsection{RBM in a Weyl Chamber}

Reflected Brownian motion in the \emph{Weyl chamber}
 \begin{align*}
 	\Omega &= \left\{x\in\R^d\colon x_1>x_2>\ldots > x_d \right\},
\end{align*}
with normal reflection is a realisation of \emph{Brownian motion with rank-dependent drift}, as studied for example
by \ts{Pal--Pitman} \cite{PalPitman08}. This is defined by taking $d$ standard Brownian motions $X_1,\ldots,X_d$ with increasing re-ordering $X_{(1)},\ldots,X_{(d)}$ and drift $\mu$ with $-\mu\in\Omega$ such that at each time $t$ the process $X_{(j)}$ has drift $\mu_j$. It can also be viewed as a Doob transform of the \emph{Delta-Bose gas}, see \ts{Prolhac--Spohn} \cite{ProlhacSpohn11}. 
The corresponding GRBM has in fact been studied before, by \ts{Rost--Vares} \cite{RostVares85}.
Since GRBM with normal reflection is a gradient diffusion it follows that this process has an invariant measure in product form 
as in Theorem~\ref{Thm:MainGeneralDomain} (see Remark~\ref{Rmk:GradientDiff}).

\section{Proofs}
\label{Sec:Proofs}

In this final section we prove Theorems \ref{Thm:Analytic} and \ref{Thm:AnalyticOrthant} and Proposition~\ref{prop:AssumptionU}. We first prove the purely analytic fact that the analogue of the basic adjoint relations hold in both the general domain (Theorem \ref{Thm:Analytic}) and the orthant (Theorem \ref{Thm:AnalyticOrthant}) and then establish, by proving Proposition~\ref{prop:AssumptionU}, a sufficient condition for $U$ to be regular with respect to a given input data $(N,Q,\mu,A)$.

\subsection{Proof of Theorem~\ref{Thm:Analytic}}

We wish to show that $\g^*p=0$ where $\g^*$ is the formal adjoint of the generator $\g$ and $p$ is as in (\ref{Eq:IMHW}) and (\ref{Eq:DensityOrthant}) respectively. Recall from the comments after (\ref{Eq:DefGenerator}) that the formal adjoint of \g\ is given by
\begin{align*}
	\g^* & = \inv{2}\,\Delta - \Om\cdot\nabla - \nabla\cdot\Om
	\intertext{where the function $\Om\colon\R^d\map\R^d$ is given by}
		\Om( x) & = \sum_{r=1}^k U'\left(n_r\cdot x- b_r\right)v_r - \mu.
	\intertext{Note further that $p$ has the form $p(x)=\exp\left\{W(x)\right\}$ where}
	W(x) & = 2\left[\sum_{r=1}^k U\left(n_r\cdot x-b_r\right) -\gamma\cdot x\right].
	\intertext{Let us remark that}
	\Om(x) & = \inv{2}\,\nabla W(x) + \left(\gamma-\mu\right)+\sum_{j=1}^k U'\left(n_r\cdot x - b_r\right) q_r.
\end{align*}
Because $q_j\cdot n_j=0$ for all $j$ we have $\nabla\cdot\Om = \inv{2}\,\Delta W$. Further $\nabla p = p\nabla W$ and therefore $\Delta p = \left(\Delta W +\abs{\nabla W}^2\right)p$. Hence it follows that
\begin{align*}
	\g^*p &= \inv{2}\Delta p - \left[\inv{2}\,\nabla W + \left(\gamma-\mu\right)+\sum_{j=1}^k U'\left(n_r\cdot x - b_r\right) q_r \right]\cdot\nabla p - \inv{2}\,W p\\
	 & = \inv{2}\,\left[\Delta W + \inv{2}\,\abs{\nabla W}^2 \right] p -\inv{2}\,\abs{\nabla W}^2 p + \left(\gamma-\mu\right) \cdot\nabla W p\\
	 &\quad\quad - \sum_{r=1}^k U'\left(n_r\cdot x- b_r\right) q_r\cdot \nabla  p - \inv{2}\,\Delta W  p.
\end{align*}
Hence,
\begin{align*}
	 \frac{\g^*p(x)}{p(x)} & = \left(\mu-\gamma\right)\cdot\nabla W(x) - \sum_{r=1}^k U\left(n_r\cdot x- b_r\right) q_r \cdot \nabla W(x)\\
		& = 2\sum_{j=1}^k U'\left(n_j\cdot x- b_j \right) \left(\mu-\gamma\right)\cdot n_j - 2 \left(\mu-\gamma\right)\cdot\gamma \\
		&\quad\quad- \sum_{r,s} U'\left(n_r\cdot x-b_r\right) U'\left(n_s\cdot x-b_s\right) n_r\cdot q_s - 2 \sum_{r=1}^k U'\left(n_r\cdot x-b_r\right) \gamma\cdot q_s\\
		& = 2\sum_{r=1}^k U'\left(n_r\cdot x-b_r\right) \left[N\mu-(N-Q)\gamma \right]_r.
\end{align*}
where we have used the skew-symmetry condition and $[\bm{y}]_r$ denotes the $r^\text{th}$ entry of a vector $\bm{y}$. The fact that the last line is equal to zero follows from the fact that $\overline{N}\mu=(\overline{N}-\overline{Q})$ for \emph{any} choice of invertible submatrix $\overline{N}$ of $N$ (and corresponding submatrix $\overline{Q}$ of $Q$) and that each row of $N$ must occur in at least one invertible submatrix, because of Lemma~\ref{Thm:LinAlg} below.

\par\noindent In the proof above we have used the following simple lemma from linear algebra.

\begin{lem}
	\label{Thm:LinAlg}
	Let $k>d$ and $\mathcal{E}=\{y_1,\ldots,y_k\}\subset\R^d$ a set of vectors whose span is the whole of $\R^d$. If there exists $r\in\ul{k}$ such that any collection of $d$ elements of $\mathcal{E}$ containing $y_r$ is linearly dependent then $y_r=0$.
	\begin{proof}
		The assumption that $\mathcal{E}$ spans $\R^d$ is equivalent to $\mathcal{E}$ containing a basis. By re-ordering we may therefore assume that $\{y_1,\ldots,y_d\}$ is such a basis. If now the set $\{y_1,\ldots,\wh{y}_j,\ldots,y_{d+1}\}$ (meaning the vector $y_r$ is left out) is linearly dependent for $r\in\ul{d}$ then $y_{d+1}$ is in the span of $\{y_1,\ldots,\widehat{y}_r,\ldots y_d\}$. By assumption this is true for \emph{every} $r\in\ul{d}$, so 
		\begin{align*}
			y_{d+1}\in\bigcap_{r=1}^d \text{span}\left\{y_1,\ldots,\widehat{y}_r,\ldots,y_d\right\}=\{0\}
		\end{align*}
		which completes the proof.
	\end{proof}
\end{lem}

\subsection{Proof of Theorem \ref{Thm:AnalyticOrthant}}

Next we turn to the proof of Theorem \ref{Thm:AnalyticOrthant}, establishing the analogue of the basic adjoint relations for GRBM in an orthant. The generator of this process is now of the form 
\begin{align*}
	\g & = \inv{2}\, \nabla\cdot A\nabla + \Om\cdot\nabla
	\intertext{where $\Om$ is given by}
			\Om(x) & = \sum_{j=1}^d U'\left(n_j\cdot x\right)\left(e_j+q_j\right) - \mu.
			\intertext{The formal adjoint of $\g$ is given by $\g^*=\inv{2}\,\nabla\cdot A\nabla-\Om\cdot\nabla-\nabla\dot\Om$, and we need to prove that $\g^*p=0$ where $p(x)=\exp\left\{W(x)\right\}$,}
			W(x) & = 2\left(\sum_{j=1}^d U\left(x_j\right) - \delta(\mu)\cdot x \right).
			\intertext{and $\delta(\mu)=\left(2A-I-Q\right)^{-1}\mu$. Noting that}
			\Om(x) & = \inv{2}\,\nabla W(x) +\sum_{j=1}^d U'\left(x_j\right) q_j + \delta(\mu)-\mu
	\end{align*}
and hence $\nabla\cdot\Omega = \inv{2}\,\Delta W$ it follows that
\begin{align*}
	\g^*p & = \inv{2}\,\left(\nabla\cdot A\nabla + \nabla W\cdot A\nabla W \right) p - \left[\left(\inv{2}\,\nabla W + \sum_{j=1}^d U'\left(n_j\cdot x\right) q_j\right)\cdot \nabla W \right]p\\
	 & \quad\quad + \left(\mu-\delta(\mu)\right)\cdot \nabla W p - \frac{p}{2}\, \Delta W. 
		\intertext{Dividing by $p$ and then using $\nabla W(x)=2\left(\sum_j U'\left(n_j\cdot x\right) n_j - \delta(\mu)\right)$ we obtain}
		\frac{\g^*p(x)}{p(x)} & = \inv{2}\,\nabla\cdot\left(A-I\right)\nabla W(x)+\inv{2}\,\nabla W(x)\cdot\nabla\left(A-I\right)\nabla W(x)\\
		& \quad\quad - \sum_{j=1}^d U'\left(x_j\right) q_j \cdot\nabla W(x) + \left(\mu-\delta(\mu)\right)\cdot\nabla W(x)\\
		& = \frac{1}{2}\,\sum_{j=1}^d \left[U''\left(x_j\right)\left(A-I\right)_{jj}\right] + 2\,\sum_{j,r} U'\left(x_j\right) U'\left(x_r\right) \left(A-I\right)_{jr}\\
			&\quad - 2\,\sum_{j=1}^d U'\left(x_j\right) \left[\delta(\mu)\cdot\left(A-I\right)e_j+e_j\cdot\left(A-I\right)\delta(\mu)\right] + 2\delta(\mu)\cdot\left(A-I\right)\delta(\mu)\\
		&\quad + 2\,\sum_{j=1}^d U'\left(x_j\right) \left[q_j\cdot\delta(\mu) + \left(\mu_j-\delta(\mu)_j\right) \right]\\
		&\quad - 2\,\sum_{j,r}U'\left(x_j\right) U'\left(x_r\right) q_{jr} +2\delta(\mu)\cdot\left(\delta(\mu)-\mu\right)
		\intertext{The first term vanishes because the diagonal entries of $A$ are all equal to 1. Using the fact that $A$ is symmetric we have}
		\frac{\g^*p(x)}{p(x)}  & = 2\sum_{j,r} U'\left(x_j\right) U'\left(x_r\right)\left[ \left(A-I\right)_{jr}-q_{jr}\right] \\
		& \quad + 2\,\sum_{j=1}^d U'\left(x_j\right) \left[q_j\cdot \delta(\mu)+\left(\mu_j-\delta(\mu)_j\right) - 2 e_j\cdot\left(A-I\right)\delta(\mu)\right]\\
		&\quad + 2\delta(\mu)\cdot\left(A-I\right)\delta(\mu) + 2\delta(\mu)\cdot\left(\delta(\mu)-\mu\right) 
\end{align*}
The first line equals zero because of the skew-symmetry condition, whereas the other two lines vanish due to the definition of $\delta(\mu)$. 

This completes the proof of Corollary~\ref{Thm:Orthant}.{\hfill\qedsymbol}

\subsection{Proof of Proposition~\ref{prop:AssumptionU}}

Finally we prove the sufficient condition for $U$ to be regular with respect to $(N,Q,\mu,A)$ stated in Proposition~\ref{prop:AssumptionU}. 

By Section 4.3 in \ts{McKean} \cite{McKean} there exists a diffusion with that generator, up to a blow-up time $\tau$, and up to $\tau$ the transition function of the diffusion has a smooth density. Assume that we have a $C^2$-function $V\colon \R^d\map\R$ such that
\begin{enumerate}[a)]
\item $V$ is proper, i.e. there exist functions $a,b\colon \R\map\R$, strictly increasing to infinity and such that $a(0)=b(0)=0$ and
\begin{align*} 
a\left(\abs{x}\right)\leq V(x) \leq b\left(\abs{x}\right)\quad\forall\, x\in\R^d
\end{align*} 
\item there exist $c>0$ and $\epsilon\geq 0$ such that $\g V \leq cV + \epsilon$ outside some compact set
\end{enumerate}
\par\noindent By Theorem 1 in \ts{Thygesen} \cite{Thygesen97} (see also \ts{Gard} \cite{Gard}, p.131), this guarantees the existence of a unique diffusion $X$ with generator $\g$, which is a \emph{regular process}, i.e. there is no blow-up in finite time, or $\tau=\infty$ almost surely. In particular the corresponding martingale problem is well-posed. Suppose now that $\rho\colon\R^d\map [0,\infty)$ is a smooth integrable function with $\g^*\rho=0$. By Proposition 9.2 of \ts{Ethier--Kurtz} \cite{EthierKurtz} it is enough to show that
\begin{align*}
	\int \left(\g f\right)(x)\,p(x)\,\d x =0
\end{align*}
for all $f\in \c_c^2(\R^d;\R)$. But this follows from $\g^*p=0$ because of compact support and integration by parts.
So the result holds as soon as we have a Lyapunov function $V$ satisfying a) and b) above. We claim that the function $V\colon\R^d\map [0,\infty)$ defined by
\begin{align*}
	V(x) & = \sum_{r=1}^k \left[\theta_{-b_r}\,n_r\cdot x - \big(U\left(n_r\cdot x - b_r\right)-U\left(-b_r\right)\big)\right]
\end{align*}
does the trick. Recall that we are only considering the cases where $A=I$ or $N=I$. Suppose the former holds. We first verify b): using the skew-symmetry condition, the fact that the $n_j$ are unit vectors and $n_j\cdot q_j=0$,
\begin{align*}
	\g V(x) & = -\inv{2}\, \sum_{r=1}^k U''\left(n_r\cdot x - b_r\right) + \sum_{r=1}^k U'\left(n_r\cdot x - b_r\right) \left[\sum_{s=1}^k \theta_{-b_s}+\mu\cdot n_r\right] \\
	&\quad\quad - \abs{\sum_{r=1}^k U'\left(n_r\cdot x-b_r\right)n_r}^2 - \sum_{r=1}^k \theta_{-b_r}\mu\cdot n_r \\
	& \leq \alpha\,\sum_{r=1}^k \left[\theta_{-b_r} \left(n_r\cdot x-b_r\right) - U\left(n_r\cdot x-b_r\right) \right] - \sum_{r=1}^k \theta_{-b_r}\mu\cdot n_r \\
	& \leq \alpha V(x) +\epsilon
\end{align*}
for some constants $\alpha, \epsilon>0$, so that b) holds. To see that $V$ is proper define $a,b\colon[0,\infty)\map\R$ by
\begin{align*}
	a(\rho) &= \inf_{\abs{x}\geq\rho} V(x),\quad\quad b(\rho) = \sup_{\abs{x}\leq \rho} V(x).
\end{align*}
Clearly the only point to prove is that $a(\rho)\longrightarrow\infty$ as $\rho\longrightarrow\infty$, but this follows from the fact that the $n_r$ form a basis and $y-U(y)\longrightarrow\infty$ as $\abs{y}\to\infty$. So a) holds which completes the proof for $A=I$.

Suppose now that $k=d$ and $N=I$. Recall that the covariance matrix $A$ of the driving Brownian motion has diagonal entries all equal to $\alpha>0$. By scaling we may assume without loss of generality that $\alpha=1$. We now define a function $V$ by
\begin{align*}
	V(x) & = \sum_{r=1}^d \left[\theta_0\, x_r-\big(U\left(x_r\right)+U(0)\big)\right].
	\intertext{So we have}
	\g V(x) & = -\frac{1}{2}\, \sum_{j=1}^d U''\left(x_j\right) + \sum_{j=1}^d \gamma_j U'\left(x_j\right) \\
	&\quad\quad - \sum_{r,j} U'\left(x_r\right) a_{rj} U'\left(x_r\right) - \sum_{j=1}^d \mu_j\theta_0\\
	& \leq \sum_{j=1}^d \left[\gamma_j U'\left(x_j\right) -\frac{1}{2}\, U''\left(x_j\right) \right] - \sum_{j=1}^d \mu_j\theta_0
\end{align*}
because $A$ is a covariance matrix, hence non-negative definite. The same argument as for the case of $A=I$ now shows that $V$ satisfies a) and b).

\ \\

\ \\\ \\ \noindent \textsc{Mathematics Institute, University of Warwick, Coventry CV4 7AL, UK}\\\ \\ \textit{Email addresses} \texttt{n.m.o'connell@warwick.ac.uk} and \texttt{j.ortmann@warwick.ac.uk}

\end{document}